\documentclass[english,a4paper,12pt]{amsart}
\usepackage[a4paper,lmargin=2cm,rmargin=2cm,tmargin=4cm,bmargin=4cm]{geometry}
\usepackage[centertags]{amsmath}
\usepackage{amsfonts}
\usepackage{amssymb}
\usepackage{amsthm}
\usepackage[draft]{graphicx}
\usepackage[colorlinks]{hyperref}

\newcommand{\Natural}{\mathbb N}

\newcommand{\abs}[1]{\left\vert#1\right\vert}
\newcommand{\set}[1]{\left\{#1\right\}}

\newcommand{\cardinality}[1]{\abs{#1}}

\newcommand{\dist}{\mathop{\mathrm{dist}}\nolimits}

\newcommand{\norm}[1]{\left\Vert#1\right\Vert}
\newcommand{\duality}[1]{\left\langle#1\right\rangle}

\newcommand{\closedball}[1]{B_{#1}}

\newcommand{\LipEmb}[1]{\displaystyle \mathop{\hookrightarrow}_{#1}}

\theoremstyle{plain}
\newtheorem{thm}{Theorem}
\newtheorem{cor}[thm]{Corollary}
\newtheorem{lem}[thm]{Lemma}
\newtheorem{prop}[thm]{Proposition}

\theoremstyle{definition}
\newtheorem{defn}[thm]{Definition}

\newtheorem{prob}[thm]{Problem}

\begin{document}
\title{Low distortion embeddings into Asplund Banach spaces}
\author{Anton\'\i n Proch\'azka$^\dag$}
\address{$^\dag$ Universit\'e Franche-Comt\'e\\
Laboratoire de Math\'ematiques UMR 6623\\
16 route de Gray\\
25030 Besan\c con Cedex\\
France}
%%%%%%%%%%%%%%
\thanks{The first named author was partially supported by  PHC Barrande 2013 26516YG}
%%%%%%%%%%%%%%%%%%%%%%%%%%
\email{antonin.prochazka@univ-fcomte.fr}

\author{Luis S\'anchez-Gonz\'alez$^\ddag$}

\address{$^\ddag$ Departamento de Ingenier{\'i}a Matem{\'a}tica\\ Facultad de CC. F{\'i}sicas y  Matem{\'a}ticas\\ Universidad de Concepci{\'o}n\\ Casilla 160-C, Concepci{\'o}n, Chile}
%%%%%%%%%%%%%%%%%
\thanks{
 The second named author was partially supported by   MICINN Project MTM2012-34341 (Spain) and FONDECYT project 11130354 (Chile). This work started while L. S\'anchez-Gonz\'alez held a post-doc position at Universit\'e Franche-Comt\'e}
%%%%%%%%%%%%%%%%%%%%%%%%5
\email{lsanchez@ing-mat.udec.cl}

\begin{abstract}
We give a simple example of a countable metric space $M$ that does not embed bi-Lipschitz with distortion strictly less than 2 into any Asplund space. Actually, if $M$ embeds with distortion strictly less than 2 to a Banach space $X$, then $X$ contains an isomorphic copy of $\ell_1$. We also show that the space $M$ does not embed with distortion strictly less than $2$ into $\ell_1$ itself but it does embed isometrically into a space that is isomorphic to $\ell_1$.
\end{abstract}
\maketitle
\section{Introduction}
We say that a Banach space $X$ is \emph{$D$-bi-Lipschitz universal} if every separable metric space embeds into $X$ with distortion at most $D$. The results of \cite{KL}, resp. \cite{Aharoni}, show that $c_0$ is $2$-bi-Lipschitz universal, resp. is not $D$-bi-Lipschitz universal for any $D<2$. In the recent preprint~\cite{Baudier}, F.~Baudier raised the following question: given a $C(K)$ Banach space $X$, what is the least constant $D$ such that $X$ is $D$-bi-Lipschitz universal? 

The goal of the present article is to prove that Asplund $C(K)$ spaces and, more generally, all Asplund Banach spaces can be $D$-bi-Lipschitz universal only if $D\geq 2$.

We obtain our $C(K)$ result as a slight modification of Baudier's elaboration on  Aharoni's original argument that $\ell_1$ does not embed into $c_0$ with distortion strictly less than $2$.
The general result is then obtained by a direct application of the deep ``Zippin's lemma''~\cite[Theorem~1.2]{Zippin}.
The general result follows also from a stronger observation: if $D<2$ and $X$ is $D$-bi-Lipschitz universal, then $\ell_1 \subset X$, which we prove in the Appendix. We also show in the Appendix that although $M$ does not embed into $\ell_1$ with distortion strictly less than $2$, it embeds isometrically into a space that is isomorphic to $\ell_1$.

An immediate corollary of our $C(K)$ result is that the $C(K)$ spaces which are $D$-bi-Lipschitz universal with $D<2$ are isometrically universal for the class of separable Banach spaces. One may ask whether this is true in general.
\begin{prob}
 Assume that $X$ is $D$-bi-Lipschitz universal with $D<2$. Does then every separable Banach space linearly embed into $X$? Does at least $c_0$ linearly embed into $X$?
\end{prob}

Let us also mention that some refined results concerning low distortion embeddings between $C(K)$ spaces can be found in our paper~\cite{ProSan}.

%\section{Preliminaries}
%All Banach spaces considered in this paper are real.  

The notation we use is standard. A mapping $f:M \to N$ between metric spaces $(M,d)$ and $(N,\rho)$ is bi-Lipschitz if there are constants $C_1,C_2>0$ such that
$C_1 d(x,y)\leq \rho(f(x),f(y))\leq C_2 d(x,y)$ for all $x,y \in M$. The distortion $\dist(f)$ of $f$ is defined as $\inf \frac{C_2}{C_1}$ where the infimum is taken over all constants $C_1,C_2$ which satisfy the above inequality. We say that $M$ embeds bi-Lipschitz into $N$ with distortion $D$ if there exists such $f:M\to N$ with $\dist(f)=D$. In this case, if the target space $N$ is a Banach space, we may always assume (by changing $f$) that $C_1=1$.
For the following notions and results, see~\cite{DGZ}. %\cite[Chapter VI.8]{DGZ}. 
A Banach space $X$ is called Asplund if every closed separable subspace $Y \subset X$ has separable dual.
A Hausdorff compact $K$ is called scattered if there exists an ordinal $\alpha$ such that the Cantor-Bendixson derivative $K^{(\alpha)}$ is empty. 
A countable Hausdorff compact is necessarily scattered. 
If $K$ is a Hausdorff compact then $C(K)$ is Asplund iff $K$ is scattered. 

\section{Results}
Let $M=\set{\emptyset} \cup \Natural \cup F$ where $F=\set{A\subset \Natural:1\leq \cardinality{A}<\infty}$ is the set of all finite nonempty subsets of $\Natural$.
We put an edge between two points $a,b$ of $M$ iff
$a=\emptyset$ and $b\in \Natural$ or $a\in \Natural$, $b \in F$ and $a\in b$ thus introducing a graph structure on $M$.
The shortest path metric $d$ on $M$ is then given for $n \neq m \in \Natural\subset M$ and $A\neq B \in F$ by 
\begin{center}
\begin{tabular}{cccc}
 $d(\emptyset,n)=1,$& $d(n,m)=2,$ & $d(n,A)=1$ if $n\in A,$ & $d(n,A)=3$ if $n \notin A,$\\
 $d(\emptyset,A)=2,$& $d(A,B)=2$ if $A\cap B \neq \emptyset,$& $d(A,B)=4$ if $A \cap B =\emptyset$.
\end{tabular}
\end{center}

%\[
%\begin{split}
%d(\emptyset,n)=1 &\mbox{ if } n \in \Natural\\
%d(n,m)=2 &\mbox{ if } n \neq m \in \Natural\\
%d(n,A)=1 &\mbox{ if } n \in A \in F\\
%d(n,A)=3 &\mbox{ if } n \notin A \in F\\
%d(\emptyset,A)= 2 &\mbox{ if } A \in F\\
%d(A,B)=0 &\mbox{ if } A=B \in F\\
%d(A,B)=2 &\mbox{ if } A\cap B \neq \emptyset \mbox{ and } A\not= B\\
%d(A,B)=4 &\mbox{ if } A\cap B = \emptyset
%\end{split}
%\]
\noindent
Thus $(M,d)$ is a countable (in particular separable) metric space.

\begin{lem}\label{l:CK}
Let $X=C(K)$ for some compact space $K$ and assume that there exist $D\in [1,2)$ and $f:M \to X$ such that 
\[
 d(x,y) \leq \norm{f(x)-f(y)}\leq Dd(x,y).
\]
Then $K$ is not scattered.
\end{lem}
\begin{proof}
We may assume, without loss of generality, that $f(\emptyset)=0$. We will show for any ordinal $\alpha$ that $K^{(\alpha)} \neq \emptyset$.
Let $\eta=4-2D>0$.
For any $i,j \in \Natural$ let $X_{i,j}=\set{x^* \in K: \abs{\duality{x^*,f(i)-f(j)}}\geq \eta}$. These are closed subsets of $K$. 
%(In other words they are $w^*$-closed subsets of the dual sphere.)
We will show that for any disjoint $A,B \in F$ and any ordinal $\alpha$ we have
\[
 \bigcap_{a \in A,b\in B} X_{a,b} \cap K^{(\alpha)} \neq \emptyset.
\]
Let us start with $\alpha=0$.
Let $A,B \in F$ be disjoint. We take $x^* \in K$ such that $\abs{\duality{x^*,f(A)-f(B)}}=\norm{f(A)-f(B)}\geq 4$.
Then for any $a \in A$ and any $b\in B$ we have
\[
\begin{split}
 \abs{\duality{x^*,f(a)-f(b)}}&\geq \abs{\duality{x^*,f(A)-f(B)}}-\abs{\duality{x^*,f(A)-f(a)}}-\abs{\duality{x^*,f(B)-f(b)}}\\
&\geq 4-2D=\eta
\end{split}
\]
thus $\displaystyle x^* \in \bigcap_{a \in A,b\in B} X_{a,b} \cap K$.

Let us assume that we have proved the claim for every $\beta<\alpha$.
If $\alpha$ is a limit ordinal then the 
%weak$^*$-
closedness of $X_{i,j}$ implies the claim.
Let us assume that $\alpha=\beta+1$.  
Let us fix two disjoint sets $A,B \in F$. Let $N=1+\max A\cup B$.
By the inductive hypothesis we know that for $N\leq i<j$ there is $x^*_{i,j} \in K$ s.t.
\[
 x^*_{i,j}\in K^{(\beta)}\cap \bigcap_{a\in A \cup \set{i} ,b \in B \cup \set{j}} X_{a,b}.
\]
We put $\Gamma:=\set{x^*_{i,j}:N\leq i<j} \subset K$.
We define $\Phi:\Natural \cap [N,\infty) \to \ell_\infty(\Gamma)$ by $\Phi(i):=(\duality{\gamma,f(i)})_{\gamma \in \Gamma}$.
Then the image of $\Phi$ is an $\eta$-separated countably infinite bounded set.
Indeed $\norm{\Phi(i)}_\infty \leq \norm{f(i)} \leq Dd(i,\emptyset)=D$.
Let $N\le i<j$ then $\norm{\Phi(i)-\Phi(j)}_{\infty} \geq \abs{\duality{x^*_{i,j},f(i)-f(j)}}\geq \eta$.
Thus $\Gamma$ is infinite and therefore $\displaystyle K^{(\beta)}\cap \bigcap_{a \in A,b\in B} X_{a,b}$ is infinite, too.
Now the 
%weak$^*$-
closedness of $\displaystyle \bigcap_{a \in A,b\in B} X_{a,b}$ and compactness of $K^{(\beta)}$ imply our claim.
\end{proof}

\begin{cor}
 If $X=C(K)$ for some compact space $K$ and $(M,d)$ embeds into $X$ with distortion strictly less than $2$, then $X$ is isometrically universal for all separable spaces.
\end{cor}
\begin{proof}
Since $C(K)$ is not Asplund, $K$ is not scattered. By a result of Pe\l czy\'nski and Semadeni~\cite{PS} there is a continuous surjection of $K$ onto $[0,1]$. Thus $C(K)$ contains isometrically $C([0,1])$ as a closed subspace. The proof is thus finished by the application of Banach-Mazur theorem~\cite{AlbiacKalton}.
\end{proof}

\begin{thm}\label{t:Main}
 Let $X$ be an Asplund space and assume that $(M,d)$ embeds into $X$ with distortion~$D$. Then $D\geq 2$. Consequently, no Asplund space is universal for embeddings of distortion strictly less than $2$ for all separable metric spaces.
\end{thm}
Since $M$ is countable, we may assume that $X$ is separable.
The proof is then based on Lemma~\ref{l:CK} and the following theorem of Zippin, see~\cite[Theorem~1.2]{Zippin} or \cite[Lemma~5.11]{Rosenthal}.
%\begin{thm}  Let $X$ be a separable Asplund space and $\frac12>\varepsilon>0$. Then there exist a compact $K$, an ordinal $\beta < \displaystyle \omega^{Sz(X,\frac\varepsilon8)+1}$, a subspace $Y$ of $C(K)$, isometric to $C([0,\beta])$ and an embedding $i:X \to C(K)$ with $\norm{i}\norm{i^{-1}}<1+\varepsilon$ such that for any $x\in X$ we have \[  \dist(i(x),Y)\leq 2\varepsilon\norm{i(x)}. \]
%The role of $C(K)$ is purely ``ambiental'' --  it should be replaced by just a Banach space $Z$. \end{thm}
\begin{thm}
 Let $X$ be a separable Asplund space. Let $\varepsilon>0$. Then there exist a Banach space~$Z$, a countable Hausdorff (in particular scattered) compact $S$, a subspace $Y$ of $Z$ isometric to $C(S)$ and a linear embedding $i:X \to Z$ with $\norm{i}\norm{i^{-1}}<1+\varepsilon$ such that for any $x\in X$ we have 
\[  
\dist_Z(i(x),Y)\leq \varepsilon\norm{i(x)}_Z. 
\] 
\end{thm}

\begin{proof}[Proof of Theorem \ref{t:Main}]
Let us assume that $D<2$.
Let $\varepsilon>0$ be small enough so that $D'=D(1+\varepsilon)<2$ and also that for $\eta:=\varepsilon2D'$ we have $\displaystyle \frac{1+4\varepsilon }{1-2\eta}D'<2$.
Then $(M,d)$ embeds into $Z$ with distortion $D'<2$ via some embedding $g$ such that $d(x,y)\leq \norm{g(x)-g(y)}\leq D'd(x,y)$. We may assume, without loss of generality, $g(\emptyset)=0$. Thus for every $x \in M$ we have $\norm{g(x)}\leq 2D'$.
%Let $\varepsilon>0$ small enough so that for $\eta:=2\varepsilon2D'$ we have $\displaystyle \frac{1+8\varepsilon }{1-2\eta}D'<2$.
We know that for each $x \in M$ there is $f(x) \in Y$ such that $\norm{g(x)-f(x)}\leq \eta$.
This implies that $\norm{g(x)-g(y)}-2\eta\leq \norm{f(x)-f(y)}\leq \norm{g(x)-g(y)}+2\eta$.
Now since $1\leq d(x,y)$ we have 
\[
 d(x,y)(1-2\eta)\leq \norm{f(x)-f(y)}\leq d(x,y)D'(1+4\varepsilon).
\]
This proves that $f$ is a bi-Lipschitz embedding of $M$ into $C(S)$ with distortion strictly less than $2$ which is impossible according to Lemma~\ref{l:CK}.
\end{proof}

\section{Appendix}
\begin{thm}\label{t:l1}
Let $X$ be a Banach space and suppose that there exist $D\in [1,2)$ and $f:M\to X$ such that $d(x,y)\leq \norm{f(x)-f(y)}\leq Dd(x,y)$. Then $X$ contains a copy of $\ell_1$.
\end{thm}
\begin{proof}
We plan to use the Rosenthal theorem~\cite{AlbiacKalton}. Thus, we have to find a sequence $(x_k) \subset X$ such that none of its subsequences is weakly Cauchy.
%Let us put it in terms a mathematician can understand: for every $(k_n) \subset \Natural$ there exist $\varepsilon>0$ and $x^*\in X^*$ such that for every $n_0$ there are $n,m \geq n_0$ such that $\duality{x^*,x_{k_n}-x_{k_m}}\geq \varepsilon$.
We claim that if we put $x_k:=f(k)$, then $(x_k)$ will have this property. 
Indeed, let $(k_n) \subset \Natural$ be given. 
We define $A_N=\set{k_{2n}:n\leq N}$ and $B_N=\set{k_{2n-1}:n\leq N}$ for every $N\in \Natural$.
Let us put $\varepsilon:=4-2D>0$.
Similarly as in the proof of Lemma~\ref{l:CK} we will put
\[
 X_{a,b}=\set{x^* \in \closedball{X^*}:\duality{x^*,f(a)-f(b)}\geq \varepsilon},
\]
and we will show that for every $n \in \Natural$
\[
 K_n:=\bigcap_{a \in A_n,b\in B_n} X_{a,b} \neq \emptyset.
\]
%(This is the usual triangle inequality+Hahn-Banach+ $D<2$.)
Now observe that $(K_n)$ is a decreasing sequence of non-empty $w^*$-compacts. Thus there exists $x^* \in \bigcap_{n=1}^\infty K_n$.
%Let $n_0 \in \Natural$ be given, we put $n=n_0$ and $m=n+1$. There exists $N$ such that $n \in A_N$ and $m \in B_N$ and $x^* \in K_N$. 
It is clear that $\duality{x^*,x_{k_{2n}}-x_{k_{2n+1}}}\geq \varepsilon$ for all $n \in \Natural$. Thus $(x_{k_i})_{i=1}^\infty$ is not weakly Cauchy.
\end{proof}

Weirdly, $M$ does not embed well into $\ell_1$ either. We will see this using Enflo's generalized roundness. Let us recall the definition as presented in~\cite{LTW}.
\begin{defn}
A metric space $(X,d)$ is said to have \emph{generalized roundness} $q$, written $q \in gr(X,d)$, if for every $n\geq 2$ and all points $a_1,\ldots,a_n,b_1,\ldots,b_n \in X$ we have
\[
\sum_{1\leq i<j\leq n}\left((d(a_i,a_j)^q+d(b_i,b_j)^q\right) \leq \sum_{1\leq i,j\leq n} d(a_i,b_j)^q.
\]
\end{defn}

\begin{prop}\label{p:roundness}
If $f:(M,d) \LipEmb{D} (X,\delta)$ and there is $0<q \in gr(X,\delta)$, then $D\geq 2$.
\end{prop}
\begin{proof}
Let $a_1,\ldots,a_n,b_1,\ldots,b_n \in M$. We have
\[
\sum_{1\leq i<j\leq n}\left((\delta(f(a_i),f(a_j))^q+\delta(f(b_i),f(b_j))^q\right) \leq \sum_{1\leq i,j\leq n} \delta(f(a_i),f(b_j))^q
\]
and so
\[
\sum_{1\leq i<j\leq n}\left((d(a_i,a_j)^q+d(b_i,b_j)^q\right) \leq D^q\sum_{1\leq i,j\leq n} d(a_i,b_j)^q.
\]
If $a_1,\ldots,a_n$ are arbitrary in the 2nd floor and $b_i=\set{a_1,\ldots,a_n}\setminus\set{a_i}$ (in the 3rd floor), the above inequality evaluates to
\[
n(n-1)2^q\leq D^q n\left((n-1)+3^q\right),
\]
i.e. $\left(\frac2D\right)^q\leq 1+\frac{3^q}{n-1}$, which is possible for all $n$ only if $D\geq 2$.
\end{proof}

\begin{cor}\label{c:BadEmbedL1}
The metric space $M$ does not embed with distortion strictly less than $2$ into any $L_1(\mu)$. 
\end{cor}
\begin{proof}
This follows immediately from Proposition~\ref{p:roundness} and from the fact that $1 \in gr(L_1(\mu))$ which is proved in \cite[Corollary 2.6]{LTW}.
\end{proof}

%\noindent
On the other hand $M$ lives isometrically in a space that is isomorphic to $\ell_1$.

\begin{prop}\label{p:l1FreeNorm}
There is an equivalent norm $\abs{\cdot}$ on $\ell_1$ such that $M$ embeds isometrically into $(\ell_1,\abs{\cdot})$.
\end{prop}
\begin{proof}
For the definition and basic properties of Lipschitz-free spaces see~\cite{GLZ} and the references therein.
The space $M$ embeds isometrically into $\mathcal F(M)$ where $\mathcal F(M)$ is the Lipschitz-free space over $M$. The result thus follows from the following fact.

\noindent
{\bf Fact:} If $(U,d)$ is a countable uniformly discrete bounded metric space then $\mathcal F(U)$ is isomorphic to $\ell_1$. \\
First, it is easy to observe that if we equip $\Natural_0=\set{0} \cup \Natural$ with the distance $\rho(0,x)=1$ for every $n \in \Natural$, and $\rho(x,y)=2$ for every $x\neq y \in \Natural$, then $\mathcal F(\Natural_0)$ is isometrically isomorphic to $\ell_1$. 
Second, it is clear that $(U,d)$ is Lipschitz homeomorphic to $(\Natural_0,\rho)$. 
The free spaces $\mathcal F(U)$ and $\mathcal F(\Natural_0)$ are thus linearly isomorphic. 
This finishes the proof of the fact and of the proposition.
\end{proof}

We do not know whether the Banach-Mazur distance between $\mathcal F(M)$ and $\ell_1$ is $2$. The above proof only shows that it is at most $4$ while Corollary~\ref{c:BadEmbedL1} shows that it is at least $2$.

\end{document}